\documentclass[11pt,a4paper]{article}
\usepackage{amsmath} 
\usepackage{amssymb}
\usepackage{amsthm}
\usepackage{dsfont}
\newtheorem{theorem}{Theorem}[section]
\newtheorem{lemma}[theorem]{Lemma}
\newtheorem{proposition}[theorem]{Proposition}
\newtheorem{corollary}[theorem]{Corollary}
\newtheorem{definition}[theorem]{Definition}
\newtheorem{remark}{Remark}
\newcommand{\R}{\mathbb{R}}
\newcommand{\N}{\mathbb{N}}

\newcommand{\bJ}{\mathbf{J}}
\newcommand{\bK}{\mathbf{K}}

\begin{document}
%\begin{frontmatter}

\title{Cell polarisation model : the 1D case}
%\footnote{This work was done during the first author's visit to
 %Paris Descartes University. This work was supported by ANR program JCJC project MODPOL}
\author{Thomas Lepoutre \footnote{INRIA Rh\^one Alpes (team DRACULA)
%Project-team DRACULA,} 
Batiment CEI-1,
66 Boulevard NIELS BOHR, 69603 Villeurbanne cedex, France \newline
Universit\'e de Lyon
CNRS UMR 5208
Universit\'e Lyon 1
Institut Camille Jordan
43 blvd. du 11 novembre 1918
F-69622 Villeurbanne cedex
France\newline
\texttt{email: thomas.lepoutre@inria.fr}
}, Nicolas Meunier \footnote{MAP5, CNRS UMR 8145, Universit\'{e} Paris Descartes, 45 rue des Saints  P\`{e}res, 75006 Paris, France \newline
\texttt{email: nicolas.meunier@parisdescartes.fr, nicolas.muller@parisdescartes.fr}} and Nicolas Muller\footnotemark[2]}%{MAP5, CNRS UMR 8145, Universit\'{e} Paris Descartes, 45 rue des Saints  P\`{e}res, 75006 Paris, France}}
%\corref{cor1}}
%\email{thomas.lepoutre@inria.fr}

%\author[map5]{Nicolas Meunier}
%\corref{cor2}}
%\ead{nicolas.meunier@parisdescartes.fr}

%\corref{cor3}}
%\ead{nicolas.muller@parisdescartes.fr}

%\address[inria]{INRIA R\^one Alpes and Institut Camille Jordan, Universit\'e Claude Bernard Lyon 1}
%\address[map5]{MAP5, CNRS UMR 8145, Universit\'{e} Paris Descartes, 45 rue des Saints  P\`{e}res, 75006 Paris, France}

%\cortext[cor2]{Corresponding author}
%\cortext[cor3]{Corresponding author}
%
%\author{Thomas Lepoutre \thanks{Institut Camille Jordan, Universit\'e Claude Bernard Lyon 1, ({\tt thomas.lepoutre@inria.fr})} \and Nicolas Meunier\thanks{MAP5, CNRS UMR 8145, Universit\'{e} Paris Descartes, 45 rue des Saints  P\`{e}res
%75006 Paris,
%France. ({\tt nicolas.meunier@parisdescartes.fr})} \and Nicolas Muller \thanks{MAP5, CNRS UMR 8145, Universit\'{e} Paris Descartes, 45 rue des Saints  P\`{e}res
%75006 Paris,
%France. ({\tt nicolas.muller@parisdescartes.fr})}.
%        }

\maketitle
\begin{abstract}
We study the dynamics of a one-dimensional non-linear and non-local drift-diffusion equation set in the half-line, with the coupling involving the trace value on the boundary. The initial mass $M$ of the density determines the behaviour of the equation: attraction to self similar profile, to a steady state of finite time blow up for supercritical mass.
Using the logarithmic Sobolev and the HWI inequalities we obtain a rate of convergence for the cases subcritical and critical mass. Moreover, we prove a comparison principle on the equation obtained after space integration. This concentration-comparison principle allows proving blow-up of solutions for large initial data without any monotonicity assumption on the initial data. 
\end{abstract}

%\begin{AMS}
%35B60; 35B44; 35Q92; 92C17; 92B05.
%\end{AMS}

\section{Introduction}

In this paper we improve the analysis of a one-dimensional non-linear and non-local convection-diffusion equation introduced in a previous paper \cite{Siam_CHMV}
\begin{equation}\label{main_0}
\left\{
\begin{aligned}
\partial_t n(t,x)-\partial_{xx}n(t,x) &=n(t,0)\partial_x n(t,x), \quad x>0,\\
\partial_x n(t,0) + n(t,0)^2 &=0,\\
n(0,x) & = n^0 (x),  \quad x>0.
\end{aligned}
\right.
\end{equation}
The boundary condition ensures mass conservation:
\begin{equation}\label{M}
\int_0^\infty n(t,x)dx=\int_0^\infty n^0(x)dx=M.
\end{equation}

The previous model \eqref{main_0} comes from the models given in \cite{Firstpaper} and \cite{Siam_CHMV} under a radial symmetry assumption. These latter models describe cell polarisation.  Cell polarisation is a major step involved in several important cellular processes such as directional migration, growth, oriented secretion, cell division, mating or morphogenesis. When a cell is not polarised proteins Cdc42 are uniformly distributed on the membrane while polarisation is characterized by a concentration of proteins in a small area of the cell membrane. 

In \cite{Firstpaper} and \cite{Siam_CHMV}, in dimension higher than two, a class of non-linear convection-diffusion models were designed, and studied, for cell polarisation. In these models there is a coupling between the evolution of proteins and the dynamics of the cytoskeleton: the proteins  diffuse and they are actively transported along tubes or filaments towards the membrane. The advection field is obtained through a coupling with the concentration of markers on the membrane.  The resulting motion is a biased diffusion regulated by the markers themselves.

Of special interest is the fact that solutions of \eqref{main_0}  may become unbounded in finite time (so-called blow-up). Blow-up of solution of \eqref{main_0} means that convection wins over diffusion. In such a situation markers concentrate on the membrane of the cell. We recall this result from \cite{Siam_CHMV}.
\begin{theorem}\label{th:1D BU} Assume $M>1$. Any weak solution of equation \eqref{main_0} (in the sense of Definition \ref{def:weak})  with non-increasing initial data $n^0$ blows-up in finite time. 
\end{theorem}

In this paper, we state a so-called concentration-comparison principle on the equation obtained after space integration. This principle together with the use of a self-similar supersolution gives estimates from above on $n(t,0)$ and a comparison with a suitable heat equation allow obtaining estimates from below on $n(t,0)$ and extending the blow-up result to any initial data above the critical mass.
\begin{proposition}[Concentration-comparison principle]\label{CCP}
Let $n_1,n_2\in C^1(0,T;\R_+)$ be two solutions of equation \eqref{main_0} satisfying 
$$
N_1^0(x)=\int_0^x n_1^0(y)dy\geq N_2^0(x)=\int_0^x n_2^0(y)dy, \quad \forall x\geq 0,
$$ 
then, we have 
$$
\forall t>0,\forall x\geq 0, N_1(t,x)=\int_0^x n_1(t,y)dy\geq N_2(t,x)=\int_0^x n_2(t,y)dy.
$$
\end{proposition}

On the other hand, global existence and asymptotic behaviour in the sub-critical and critical cases, $M\leq1$, were established in \cite{Siam_CHMV}.
\begin{theorem}\cite{Siam_CHMV} \label{th:1D} 
Assume that the initial datum $n^0$ satisfies both $n^0 \in L^1(( 1 + x)d x)$ and \\
$\int_0^\infty n^0(x) (\log n^0(x))_+ d x<+\infty$. Assume in addition that $M\leq 1$, then there exists a global weak solution that satisfies the following estimates for all $T>0$,
\begin{eqnarray*}
\sup_{t\in (0,T)} \int_0^\infty n(t,x) (\log n(t,x))_+ d x &<& +\infty\, ,\\
\int_0^T\int_0^\infty n(t,x) \left( \partial_x \log n(t,x) \right)^2 d x d t &<& +\infty\, . 
\end{eqnarray*}
In the sub-critical case $M<1$ the solution strongly converges in $L^1$ towards the self-similar profile $G_\alpha $ given by (\ref{eq:stat state rescaled}) in the following sense: 
\[ \lim_{t\to +\infty }\left\|n(t,x) - \frac{1}{\sqrt{1+ 2t}} G_\alpha \left(\frac{x}{\sqrt{1+ 2t}}\right) \right\|_{L^1} = 0\, .  \]
In the critical case $M = 1$, assuming in addition that the second momentum is finite $\int_0^\infty x^2n^0(x)dx<+\infty$, the solution strongly converges in $L^1$ towards the unique stationary state $\nu_\alpha (x):=\alpha \exp(-\alpha x)$, where $\alpha^{-1} = \int_0^\infty x n^0(x)d  x$.
\end{theorem}

In \cite{Siam_CHMV}, in  the sub-critical case, the asymptotic result was obtained through the convergence to zero of a suitable Lyapounov functional  $L$ defined by \eqref{Lyapounov_ss_crit}.  Here, using a logarithmic Sobolev inequality with a suitable function, we obtain an exponential decay to equilibrium in self similar variables replacing in particular the former result by 
\begin{proposition}\label{rate}
Under the assumptions of theorem~\ref{th:1D}, we have
$$
\left\|n(t,x) - \frac{1}{\sqrt{1+ 2t}} G_\alpha \left(\frac{x}{\sqrt{1+ 2t}}\right) \right\|_{L^1}\leq \frac{C}{(1+2t)^{3/2}}.
$$
\end{proposition}
Finally, in the critical case $M = 1$,  we actually improve the result given in Theorem \ref{th:1D}  by precising the speed of convergence. Firstly, we  give a rate. Secondly, using the HWI inequality, we improve the rate of convergence when the third momentum is initially finite.
 
We end this introductory Section with one open question that we are not able to resolve: obtain a blow-up profile for large initial datum.

The plan of this work is the following. First,  we state a concentration-comparison principle and the finite time blow-up for supercritical mass.  Then, we  give a quantitative argument to the systematic blow-up for supercritical mass. In a third step we study the sub-critical mass and we give a rate to the self similar decay. Finally, we study the critical mass, by using first a Lyapounov approach and then the HWI inequality.

%The logarithmic Sobolev and the HWI inequalities are usually used for entropy dissipation methods, see e.g. \cite{Arnold} and \cite{TOT}. In this work, we have used these inequalities in order to get better estimations on speed of convergence for the non linear 1D convection-diffusion model. A first improvement has been done in the critical case by bounding in a better way the Wasserstein distance. A second improvement would rely on better estimations on the second momentum. We cannot have comparison principle on the solutions of the 1D polarisation model. However, the concentration-comparison principle on the integrated equation allowed us establishing a new way to study this latter equation. We can also add a convection term in the heat equation for sur-critical case in order to improve estimation of blow-up time. We leave the blow-up profile for further work.

\section{A brief account of some useful facts}
We first recall some facts concerning \eqref{main_0}, see \cite{Siam_CHMV} for more details. Let $n$ be a classical solution of \eqref{main_0} on $(0,T)$, a straightforward computation of the evolution of the entropy yields  
\begin{eqnarray}
  \frac{ d }{ d t} \int_0^\infty n(t,x) \log n(t,x) dx & = &- \int _0^\infty \left(\partial_x n (t,x)  + n(t,0) n(t,x)\right) \frac{\partial_x n(t,x) }{n(t,x)}  d x \nonumber \\  
& = & -  \int _0^\infty  n(t,x)\left( \partial_x \log n(t,x) \right)^2 d x    +    n(t,0) ^2 \, . 
\label{eq:entropy dissipation}
\end{eqnarray}

Moreover a proper definition of weak solutions, adapted to our context is 
\begin{definition}\cite{Siam_CHMV} \label{def:weak}
We say that $n(t,x)$ is a weak solution of \eqref{main_0}  on $(0,T)$ if it satisfies:
\begin{equation*}
n \in L^\infty(0,T;L^1_+(\R_+))\, , \quad \partial_x n \in L^1((0,T)\times \R_+)  \, , \label{eq:flux L1}
\end{equation*}
and $n(t,x)$ is a solution of \eqref{main_0} in the sense of distributions in $\mathcal D'(\R_+)$.
\end{definition}

Since the flux $\partial_x n(t,x) + n(t,0) n(t,x)$ belongs to $ L^1((0,T)\times \R_+)$, the solution is well-defined in the distributional sense under the assumptions of definition (\ref{eq:flux L1}). In fact we can write
$\int_0^T n(t,0) d t   = - \int_0^T\int_0^\infty \partial_x n(t,x) d x d t$.

From the maximum principle \cite{Evans} it follows that $n\ge 0$ if $n^0\ge 0$. Furthermore weak solutions in the sense of Definition \ref{def:weak} are mass-preserving: 
\[ M =\int_0^\infty n^0 (x) d x = \int_0^\infty n (t,x) d x.\]

Throughout the paper, for $n\geq 0\in L^1(\R_+)$, we define the generalized distribution function in the following way:
\begin{equation*}
N(x)=\int_0^x n(y)dy,
\end{equation*}
and for any weak solution $n$ in the sense of definition \ref{def:weak}, the first momentum is defined as
\begin{equation*}
J(t) = \int_0^\infty x n(t,x) dx.
\end{equation*}
Let $u(x) dx$ and $\nu(x) dx$ be measures with smooth densities on $\R_+$, we note the relative entropy
\begin{equation*}
H(u|\nu) = \int_0^\infty u(x) \log \left(\frac{u(x)}{\nu(x)} \right) dx
\end{equation*}
and the Fisher information
\begin{equation*}
I(u|\nu) = \int_0^\infty  u(x)\left(\partial_x\left(\log \frac{u(x)}{\nu(x)}\right)\right)^2 dx.
\end{equation*}
We also define the quadratic Wasserstein distance between two probability measures $\mu$ and $\nu$ on $\mathbb{R}_+$ with finite second momentum as in \cite{TOT},
\begin{equation}\label{defWasserstein}
W_2(\mu,\nu) = \inf_{\pi \in \Pi(\mu,\nu)} \sqrt{ \iint_{\mathbb{R}_+ \times \mathbb{R}_+} |x-y|^2 d\pi(x,y)}.
\end{equation}
where $\Pi(\mu,\nu)$ denotes the set of probability measures on $\mathbb{R}_+ \times \mathbb{R}_+$ with marginals $\mu$ and $\nu$, i.e. such that for all test functions $\phi$ and $\psi$ in a suitable class of test functions,
\begin{equation*}
\iint_{\mathbb{R}_+ \times \mathbb{R}_+} (\phi(x)+\psi(y)) d\pi(x,y) = \int_0^\infty \phi(x) d\mu(x)+\int_0^\infty \psi(y) d\nu(y).
\end{equation*}
%The proof of the triangle inequality of Wasserstein distance can be found in \cite{TOT}.INUTILE A MON AVIS
 In particular, if $\nu$ is a Dirac measure $\delta_a$ with $a \in \mathbb{R}_+$ we have
\begin{equation}\label{defWassersteinDirac}
W_2(\mu,\delta_a)^2 = \int _0^\infty | x-a|^2 d\mu(x).
\end{equation}
A simple computation on the first momentum (if it is initially finite) leads to 
\begin{equation}\label{dJ}
\frac{d}{dt} J(t)=n(t,0)(1-M).
\end{equation}
We recall the three situations:
\begin{itemize}
\item when $M=1$, there is a balance between drift and diffusion and we expect a convergence to a steady state (see section~\ref{sec:M=1}),
\item when $M<1$, the diffusion drives the equation and we have a self similar behaviour see section~\ref{sec:M<1},
\item when $M>1$, solution blows up in finite time see section~\ref{sec:M>1}.
\end{itemize}

\begin{remark}
Such a critical mass phenomenon (global existence {\em versus} blow-up) has been widely studied for the Keller-Segel system (also known as the Smoluchowski-Poisson system) in two dimensions of space, see \cite{BDP,P} e.g. and the references therhein. The equation (\ref{main_0}) represents in some sense a caricatural version of the classical Keller-Segel system in the half-line $(0,+\infty)$. 
\end{remark}
\begin{remark}
There is a strong connection between the equation under interest here (\ref{main_0}) and the one-dimensional Stefan problem. The later writes  \cite{HV1}:
\[\left\{\begin{array}{l}
\partial_t u(t,x) = \partial_{xx} u(t,x) \, , \quad \, t>0\, , \, x\in (-\infty,s(t))\, , \\
\lim_{x\to -\infty}\partial_xu (t,x) = 0 \, , \quad u(t,s(t)) = 0\, , \quad \partial_x u (t,s(t)) = -s'(t)\, .
\end{array}\right.\]
The temperature is initially non-negative: $u(0,x) = u_0(x)\geq 0$. By performing the following change of variables: $\phi(t,x) = - u(t,s(t)-x)$, we get an equation that is linked to (\ref{main_0}) by $n(t,x) = \partial_x \phi(t,x)$. This connection provides some insights concerning the possible continuation of solutions after blow-up \cite{HV1}. This question has raised a lot of interest in the past recent years \cite{HV2,DS}. It is postulated in \cite{HV1} that the one-dimensional Stefan problem is generically non continuable after the blow-up time. 
\end{remark}

%%%%%%%%%%%%%%%%%%%%%%%%%%%%%%%%%%%%%%%%%%%%%%%%%%%%%%%%%%%%%%%%%%%%%%%%%%%%%%%%%%%%%%%%%%%%%%%%%%%%%%%%%%%%

\section{Concentration-comparison principle and finite time blow up for supercritical mass}\label{sec:M>1}
The main result of this section is a concentration-comparison principle given by Proposition \ref{CCP} that we first prove. Then, we use this principle to compare solutions of equation \eqref{main_0}, rewritten as an equation on $N$, to both sub- and supersolutions. Finally we provide a quantitative blow-up argument by estimating the blow-up time $T^*$ defined by $\bJ(T^*)=0$.
\subsection{Proof of Proposition  \ref{CCP}}
We actually prove the following stronger lemma:
\begin{lemma}\label{lem:ccp_sub_super}
Let $N,\bar N, \underline{N}$ be nondecreasing (in space)  functions in $C^1(0,T;C^2(\R_+))$ satisfying 
$$\partial_x N,\partial_x\bar N, \partial_x\underline{N}\in L^\infty(0,T;L^\infty(\R_+))$$
and
\begin{equation}\label{main_integre}
\begin{cases}
\partial_t N(t,x)-\partial_{xx}N(t,x)-\partial_xN(t,0)\partial_xN(t,x)=0,\\[0.3cm]
\partial_t \bar N(t,x)-\partial_{xx}\bar N(t,x)-\partial_x\bar N(t,0)\partial_x\bar N(t,x)\geq 0,\\[0.3cm]
\partial_t \underline{N}(t,x)-\partial_{xx}\underline{N}(t,x)-\partial_x\underline{N}(t,0)\partial_x\underline{N}(t,x)\leq 0,\\
N(t,0)=\bar N(t,0)= \underline{N}(t,0)=0.
\end{cases}
\end{equation}
Assume that $\underline{N}(0,.)\leq N(0,.)\leq \bar N(0,.)$, then for $0<t\le T$ we have that
$$
\underline{N}(t,.)\leq N(t,.)\leq \bar N(t,.).
$$
\end{lemma}
\proof
It is obviously sufficient to prove the comparison  $\bar N\geq \underline{N}$. Firstly, we assume that  
$$
\partial_x\bar N(0,0)>\partial_x\underline{N}(0,0)\;(\geq 0).
$$
 Integrating equation \eqref{main_0} in space, we see that $\delta N=\bar N- \underline{N}$ satisfies the parabolic equation
\begin{equation*}
\begin{cases}
\partial_t \delta N(t,x)-\partial_{xx}\delta N(t,x)-\partial_x\bar N(t,0)\partial_x\delta N(t,x)=(\partial_x\bar N(t,0)-\partial_x\underline{N}(t,0))\partial_x\underline{N}(t,x),\\
N_i(t,0)=0.
\end{cases}
\end{equation*}

Since we supposed  $\partial_x\bar N(0,0)>\partial_x\underline{N}(0,0)$, it remains true at least until a time $T' \in ]0,T[$, we choose the maximal $T'$ possible. On the time interval $[0,T']$ we have, since $\partial_x\underline{N}\geq 0$,
$$
\begin{cases}
\partial_t \delta N(t,x)-\partial_{xx}\delta N(t,x)-\partial_x\bar N(t,0)\partial_x\delta N(t,x)\geq 0,\\
N_i(t,0)=0.
\end{cases}
$$
Hence $\bar N> \underline{N}$ on $]0,T'[\times \R_+^*$ by strong maximum principle \cite{Evans}. Furthermore, by Hopf Lemma (see \cite{Evans}), we also have 
$$
\partial_x \delta N(T',0)=\partial_x\bar N(T',0)-\partial_x\underline{N}(T',0)>0.
$$
As $T'$ is maximal we immediately conclude that $T'=T$.
To treat the case $\partial_x\bar N(0,0)=\partial_x\underline{N}(0,0)$, we use the following 
\begin{lemma}
Let $\bar N$ be a supersolution as in lemma~\ref{lem:ccp_sub_super}, denote $K=\sup_{[0,T]}\partial_x \bar N$, then 
$$
\bar N^\epsilon (t,x)=\bar N(t,x)+\epsilon e^{4K^2t}(1-e^{-4Kx})
$$ 
is a supersolution on $[0,T]$ in the sense of lemma~\ref{lem:ccp_sub_super} for $0\leq \epsilon<\frac{1}{2}e^{-4K^2T}$.
\end{lemma}
\proof It is straightforward that $\partial_x \bar N^\epsilon\geq 0$ and $\bar N^\epsilon(t,0)=0$. We essentially have to prove the inequation
$$
I^\epsilon=\partial_t \bar N^\epsilon-\partial_{xx}\bar N^\epsilon-\partial_x\bar N^\epsilon(t,0)\partial_x \bar N^\epsilon\geq 0.
$$
 We denote 
$$
\bar N^\epsilon (t,x)=\bar N(t,x)+\beta(t)(1-e^{-\alpha x}).
$$
Conditions of the lemma ensure $\beta(t)\leq \frac{1}{2}$ for $0\leq t\leq T$. 
Since $\bar N$ is a supersolution, we have
$$
I^\epsilon\geq
\beta'(t)(1-e^{-\alpha x})+\alpha^2\beta(t)e^{-\alpha x}-\beta(t)\alpha\partial_x N-\beta(t)\alpha e^{-\alpha x}\partial_x\bar N(t,0)-\alpha^2\beta^2e^{-\alpha x}.
$$
This leads to, by definition of $K$, 
$$
I^\epsilon\geq \beta'-\beta\alpha K+\beta e^{-\alpha x}\left(-\frac{\beta'}{\beta}+\alpha^2(1-\beta)-\alpha K\right)
$$
and the right hand side is nonnegative. We choose the $\beta(t)=\epsilon e^{\alpha Kt}$, so that 
$$
I^\epsilon\geq \beta e^{-\alpha x}\left(\alpha^2(1-\epsilon e^{\alpha K T})-2\alpha K\right),
$$
and finally, choosing $\alpha =4K$ and $\epsilon\leq \frac{e^{-4K^2T}}{2}$ we have $I^\epsilon\geq 0$.
\qed

\textit{End of the proof of lemma~\ref{lem:ccp_sub_super}}. 

Thanks to the first step of the proof, since for $\epsilon>0$ we have $\partial_x\bar N^\epsilon(t,0)>\partial_x\bar N(t,0)\geq \partial_x\underline{N}(t,0),$
we can compare $\bar N^\epsilon(t,x)\geq \underline{N}(t,x)$ for $0\leq t\leq T$ and $\epsilon $ small enough. Letting $\epsilon\rightarrow 0$ we can conclude $\bar N(t,x)\geq \underline{N}(t,x)$.\qed 

\subsection{Estimations on $n(t,0)$ and local existence for supercritical mass}
Classically, the concentration-comparison principle allows us comparing solutions of \eqref{main_integre} to sub- and supersolutions. 
Since the proof is very much alike the one for the comparison principle (also based on Hopf Lemma)  we do not repeat it. 

The main use of this Lemma is the following: 

\begin{lemma}[Self-similar supersolutions]
Let $\lambda\in \R$ and  $f\in C^2(\R_+,\R_+)$ satisfying 
\begin{equation*}
\begin{cases}
f'(x)>0, \quad \forall x \ge 0,\\
f(0)=0,\\
\lambda xf'(x)-f''(x)-f'(0)f'(x)\geq 0, \quad \forall x\geq 0.
\end{cases}
\end{equation*}
Then the function $\bar N_{f,\lambda}$ defined by
$$
\bar N_{f,\lambda}(t,x):=f\left(\frac{x}{\sqrt{1-2\lambda t}}\right),
$$
is a supersolution of \eqref{main_integre} for all $t>0$ if $\lambda\leq 0$ and for $t<\frac{1}{2\lambda}$ if $\lambda >0$. 
\end{lemma}
\begin{proof}
For $1-2\lambda t>0$ we first notice that 
$$
\frac{d}{dt}\frac{1}{\sqrt{1-2\lambda t}}=\frac{\lambda}{\sqrt{1-2\lambda t}^3}.
$$
Therefore, by construction, as long as $1-2\lambda t>0$, we see that
\begin{eqnarray*}
(1-2\lambda t)\left(\partial_t \bar N_{\lambda,f}(t,x)-\partial_{xx} \bar N_{\lambda,f}(t,x)-\partial_x \bar N_{\lambda,f}(t,0)\partial_x \bar N_{\lambda,f}(t,x)\right) &=&\\
\lambda \frac{x}{\sqrt{1-2\lambda t}}f'\left(\frac{x}{\sqrt{1-2\lambda t}}\right)-f''\left(\frac{x}{\sqrt{1-2\lambda t}}\right)-f'(0)f'\left(\frac{x}{\sqrt{1-2\lambda t}}\right)
&\geq &0.
\end{eqnarray*}
Finally, it is easy to prove that such a $f$ exists and that, for all $x\ge 0$, it verifies 
\begin{equation}\label{def_de_f}
f(x) \leq f'(0) \int_0^x \exp\left(\lambda \frac{y^2}{2}-f'(0) y\right) dy.
\end{equation}
\begin{flushright}
$\square$
\end{flushright}
\end{proof}
\begin{remark}
For $\lambda \leq 0$ we can have equality in \eqref{def_de_f}. In such a case, in the critical case $M=1$, we recover the steady state and in the sub-critical case, we recover the self similar profile, see \cite{Siam_CHMV}. For $\lambda >0$, the equality yields that $f(\infty)=+\infty$ which does not correspond to a $L^1$ derivative. Therefore a self-similar blow-up profile doesn't seem to appear.
\end{remark}

This result is particularly useful because it allows comparing  a large class of initial data to such a supersolution. Indeed, for a good choice of $f$, by comparison principle, we would be able to prove that $N(t,x)\leq \bar N_{f,\lambda}(t,x) $ for $t<\frac{1}{2\lambda}$. This will provide an a priori bound on $n(t,0)$, that is 
$$n(t,0)\leq \frac{1}{\sqrt{1-2\lambda t}}f'(0),  \textrm{ for } t<\frac{1}{2\lambda} .$$ 
Recalling now the entropy 'dissipation' \eqref{eq:entropy dissipation},  for $t<\frac{1}{2\lambda}$ and for $\int_0^\infty  n^0 (x) \log n^0(x) d x<+\infty$, we obtain the following a priori estimates
$$\sup_{(0,t)}\int_0^\infty n(t,x)(\log n(t,x))_+ d x +\int_0^t\int_0^\infty  n(t,x)(\partial_x\log n(t,x))^2 d x\leq C(t),$$
which is enough to get compactness, insuring existence of weak solutions build as in \cite{CalvezMeunier}. Details on existence will be carried out in \cite{NicoPhD}. To guarantee such a bound, we need to be able to compare, that is we need existence of $\lambda,\mu>0$ such that 
$$
\forall x>0, \int_0^x \mu \exp\left(\lambda \frac{y^2}{2}-\mu y\right)dy-N(0,x)\geq 0.
$$

%%%%%%%%%%%%%%%%%%%%%%%%%%%%%%%%%%%%%%%%%%%%%%%%%%%%%%%%%%%%%%%%%%%%%%%%%%%%%%%%%%%%%%%%%%%%%%%%%%%%%%%%%%%%

\subsection{Systematic blow up for supercritical mass : A quantitative argument} 
In \cite{Siam_CHMV}, the proof of the blow-up in supercritical case is based on the non-existence of the first momentum. For the convenience of the reader, we recall it now. Assume that $M>1$ and that $n^0$ is non-increasing, then $n(t,\cdot)$ is also non-increasing for any time it exists due to the maximum principle. In fact $v(t,x) = \partial_x n(t,x)$ satisfies a parabolic type equation without any source term, it is initially non-positive, and it is non-positive on the boundary.  Therefore $-\partial_x n(t,x)/n(t,0)$ is a probability density at any time $t>0$. From Jensen's inequality, we deduce that
\begin{equation*}
\left(\int_0^\infty x \frac{-\partial_x n(t,x)}{n(t,0)}dx\right)^2  \leq 
\int_0^\infty x^2 \frac{-\partial_x n(t,x)}{n(t,0)}dx, 
\end{equation*}
which rewrites as $M^2  \leq 2 n(t,0) \bJ(t)$. Plugging this latter inequality into the evolution of the first momentum (\ref{dJ}) yields that 
$$\bJ(t) \leq \bJ(0) + \frac{(1 - M)M^2}{2} \int_0^t \frac{1}{\bJ(s)}d s.$$
Introducing next the auxiliary function 
$$\bK(t) = \bJ(0) + (1 - M)M^2 \int_0^t \bJ(s)^{-1}d s$$ 
which is positive and satisfies 
$$\frac{d }{d t}\bK(t)  = \frac{(1 - M)M^2}{2} \frac{1}{\bJ(t)}   \leq \frac{(1 - M)M^2}{2} \frac{1}{\bK(t)},$$ we deduce that $\frac{d }{d t}\left[\bK(t)^2\right]  \leq  (1 - M)M^2$,
hence a contradiction.\\

In this work, we also use the non-existence of decreasing first momentum beyond the blow-up time $T^*$, given by $J(T^*)=0$, but the integrated solution gives us a way to find this $T^*$ by using equation (\ref{dJ}):
\begin{equation*}
\int_0^{T^*} \partial_x N(t,0) dt=\int_0^{T^*} n(t,0) dt = \frac{J(0)}{M-1}.
\end{equation*}
We are interested in using a subsolution of the parabolic equation on $N$ to find a blow-up time. A first lower bound on $N$ is given by Chebyshev inequality.
\begin{lemma}[Chebyshev]\label{Chebyshev}
For all $t\ge 0$ and for all $x$ satisfying $x\geq J(t)/M$, the following inequality holds true
\begin{eqnarray*}
N(t,x) & \geq & M - \frac{J(t)}{x}.
\end{eqnarray*}
\end{lemma}
\begin{proof}
Chebychev inequality applied to the probability distribution $n/M$ yields that
\begin{eqnarray*}
\int_x^{\infty} \frac{n(t,y)}{M} dy & \leq & \frac{1}{x M} \int_0^{\infty} y n(t,y) dy.
\end{eqnarray*}
\begin{flushright}
$\square$
\end{flushright}
\end{proof}

We next use this lower bound to define a subsolution of equation \eqref{main_integre} as the solution of a particular heat equation on a bounded domain $[0,L]$. Firstly, recalling that $\partial _x N$ and $n(t,0)$ are non negative, we deduce that  the solution $\widehat{N}$ of the classical heat equation on $\R_+$ 
\begin{equation}\label{heat}
\begin{cases}
\partial_t \widehat{N} (t,x) - \partial_{xx} \widehat{N} (t,x) = 0, \quad \textrm{ in }  \R_+^*, \\
\widehat{N}(t,0) = 0, \\
\widehat{N}^0 = N^0, \quad \textrm{ in }  \R_+,
\end{cases}
\end{equation}  
is a subsolution of equation \eqref{main_integre}. Next, using Lemma (\ref{Chebyshev}), we will consider a bounded domain. The advection to $0$ disappears in favor of a boundary term in $x=L$. This latter term depends on the decreasing first momentum and this gives a weak way to transport to $0$.
\begin{proposition}\label{control}
For $n$ solution of equation \eqref{main_0} with $M\geq1$,  for $L \geq \frac{J(0)}{M}$, we have 
\begin{equation*}
n(t,0) \geq \frac{M - \frac{J(0)}{L}}{L} \left( 1 + 2 \, \sum_{n \in \ \widehat{N}} \, (-1)^n \, \exp \left(-\left(\frac{n \, \pi}{L}\right)^2 t\right) \right).
\end{equation*}
Furthemore, for $t>\left(\frac{J(0)}{M \pi}\right)^2 \, \log(2)$, we have  
\begin{equation*}
n(t,0) \geq k > 0.
\end{equation*}
\end{proposition}
\begin{proof}
First, if $f$ solution of the bounded heat equation with Dirichlet conditions,  
\begin{equation}
\begin{cases}\label{classical_heat_bounded}
\partial_t f (t,x) - \partial_{xx} f (t,x) = 0, \quad \textrm{ in }  (0,L), \\
f(t,0) = f(t,L) =  0, \\
f^0 (x)= N^0(x), \quad \textrm{ in }  (0,L),
\end{cases}
\end{equation}
then we recall that
\begin{equation*}
f (t,x) = \sum_{n \in \N^*} \, \sin\left( \frac{n \, \pi \, x}{L}\right) \, \exp \left(-\left(\frac{n \, \pi}{L}\right)^2 t\right) f_n ,
\end{equation*}
with the Fourier coefficients
\begin{equation*}f_n = \frac{2}{L} \int_0^L N^0(y) \, \sin\left( \frac{n \, \pi \, y}{L}\right) dy.
\end{equation*}
Let $L$ be such that $L \geq \frac{J(0)}{M}$. Since the first momentum $J$ is decreasing,  from  Lemma (\ref{Chebyshev}), we deduce that $ N(t,L) \geq M - \frac{J(0)}{L}$ for all $t \geq 0$. We now define the following problem
\begin{equation}\label{heat_bounded}
\begin{cases}
\partial_t \tilde{N} (t,x) - \partial_{xx} \tilde{N} (t,x) = 0, \quad \textrm{ in }  (0,L), \\
\tilde{N}(t,0) = 0,  \quad \tilde{N}(t,L) =  M - \frac{J(0)}{L}, \\
\tilde{N}^0 (x)= N^0(x), \quad \textrm{ in }  (0,L).
\end{cases}
\end{equation}
We extend the solution $\tilde{N}$ on $[L,+\infty]$ by
\begin{equation*}
N^*(t,x) = \begin{cases}
\tilde{N}(t,x), & \mbox{ if } x\leq L, \\
M - \frac{J(0)}{L}, & \mbox{ if } x>L.
\end{cases}
\end{equation*}
From Lemma \ref{Chebyshev} together with the fact that the solution $\widehat{N}$ of the heat equation on $\R_+$ is a subsolution of equation \eqref{main_integre} we deduce that $N^*$ is a subsolution of equation \eqref{main_integre}. Using Fourier series to solve \eqref{heat_bounded}, we obtain that
\begin{equation*}
\tilde{N} (t,x) = \frac{x \, \tilde{N}(t,L)}{L} + \sum_{n \in \N^*} \, \sin\left( \frac{n \, \pi \, x}{L}\right) \, \exp \left(-\left(\frac{n \, \pi}{L}\right)^2 t\right) g_n ,
\end{equation*}
with the Fourier coefficients defined by
\begin{equation*}g_n = \frac{2}{L} \int_0^L \left(N^0(y)-\frac{y \, \tilde{N}(t,L)}{L}\right) \, \sin\left( \frac{n \, \pi \, y}{L}\right) dx = \frac{2}{L} \int_0^L N^0(y) \, \sin\left( \frac{n \, \pi \, y}{L}\right) dy + \frac{2 \, \tilde{N}(t,L)}{n \, \pi} \, (-1)^n.
\end{equation*}
By strong maximum principle \cite{Evans}, a solution of \eqref{classical_heat_bounded}, with initial datum $N^0 \geq 0$, is positive, hence
\begin{equation*}
\tilde{N} (t,x) \geq \frac{x \, \tilde{N}(t,L)}{L} + \sum_{n \in \N^*} \, \sin\left( \frac{n \, \pi \, x}{L}\right) \, \exp \left(-\left(\frac{n \, \pi}{L}\right)^2 t\right) \frac{2 \, \tilde{N}(t,L)}{n \, \pi} \, (-1)^n.
\end{equation*}
Furthermore, recalling that $N (t,0) = \tilde{N} (t,0) = 0$ and that $\tilde{N} \leq N$ on $[0, +\infty[ \times [0,L]$, we deduce that
\begin{equation*}
n(t,0) = \partial_x N (t,0) \geq \partial_x \tilde{N} (t,0) \geq \frac{\tilde{N}(t,L)}{L} \left( 1 + 2 \, \sum_{n \in \N^*} \, (-1)^n \, \exp \left(-\left(\frac{n \, \pi}{L}\right)^2 t\right) \right),
\end{equation*}
hence, for $t>\left(\frac{L}{\pi}\right)^2 \, \log(2)$, it follows that
\begin{equation*}
n(t,0) \geq \frac{\tilde{N}(t,L)}{L} \left( 1 - 2  \, \exp \left(-\left(\frac{\pi}{L}\right)^2 t\right) \right) = k(t) > 0.
\end{equation*}
\begin{flushright}
$\square$
\end{flushright}
\end{proof}

The previous lower bound on $n(t,0)$ provides us an upper bound on the blow-up time $T^*$.
\begin{corollary}
For $n$ solution of equation \eqref{main_0} with $M>1$, we have the following upper bound on blow-up time
\begin{equation*}
T^* \leq 4 \, \frac{J(0)^2}{M^2 (M-1)} \left(1 + \frac{M-1}{6} \right) .
\end{equation*}
\end{corollary}
\begin{proof}
From proposition \ref{control} it follows that 
\begin{equation*}
\frac{J(0)}{M-1} = \int_0^{T^*} n(t,0) dt \geq \frac{M - \frac{J(0)}{L}}{L} \, \left(T^* + \frac{2 \, L^2}{\pi^2} \, \sum_{n \in \N^*} \, \frac{(-1)^n }{n^2} \right).
\end{equation*}
Moreover, recalling that $\sum_{n \in \N^*} \, \frac{(-1)^n }{n^2} = - \frac{\pi^2}{12}$,  for all $L \geq \frac{J(0)}{M}$, we see that
\begin{equation*}
T^* \leq \frac{J(0)}{M-1} \frac{L}{M-\frac{J(0)}{L}} + \frac{L^2}{6}.
\end{equation*}
For $ L$  optimal, i.e. given by  $L=\frac{2 J(0)}{M}\geq \frac{J(0)}{M}$, we obtain that
\begin{equation*}
T^* \leq 4 \, \frac{J(0)^2}{M^2 (M-1)} \left(1 + \frac{M-1}{6} \right) .
\end{equation*}
\begin{flushright}
$\square$
\end{flushright}
\end{proof}
\begin{remark}
The upper bound on the blow-up time $T^*$ given in the previous proof is $4\left(1 + \frac{M-1}{6} \right)$ times bigger than the one found in \cite{Siam_CHMV}. Indeed in \cite{Siam_CHMV}, with a non-increasing initial condition, it was found that $T^* = \frac{J(0)^2}{M^2 \, (M-1)}$. In this work, convection is described by the decreasing of the first momentum, then we have only used diffusion phenomenon in equation \eqref{heat_bounded}. This could explain the difference in results.
\end{remark}

%%%%%%%%%%%%%%%%%%%%%%%%%%%%%%%%%%%%%%%%%%%%%%%%%%%%%%%%%%%%%%%%%%%%%%%%%%%%%%%%%%%%%%%%%%%%%%%%%%%%%%%%%%%%

\section{Subcritical mass and self similar decay}\label{sec:M<1}
For the case $M<1$, global existence has been proved in \cite{Siam_CHMV}. Here, we are interested in the asymptotic behaviour and self similar decay has also been exhibited in \cite{Siam_CHMV}. If we perform the following change of variable:
$$
n(t,x)=\frac{1}{\sqrt{1+2t}}u\left(\log(1+2t),\frac{x}{\sqrt{1+2t}}\right),
$$
then, the density $u(\tau,y)$ satisfies 
\begin{equation*}
\partial_\tau u(\tau,y)-\partial_{yy}u(\tau,y)-\partial_{y}(yu(\tau,y))-u(\tau,0)\partial_yu(\tau,y)=0,
\end{equation*}
together with a zero flux boundary condition, $\partial_y u(\tau,0)+u(\tau,0)^2=0$. The additional left-sided drift contributes to confine the mass in the new frame $(\tau,y)$. It has been proved in \cite{Siam_CHMV} that $u$ converges to $G_\alpha$ in large time, where $G_\alpha$ is given by 
\begin{equation} \label{eq:stat state rescaled}
G_\alpha(y)=\alpha \exp\left(-\alpha y-\frac{y^2}{2}\right),\quad \int_0^\infty G_\alpha(y)dy=\int_0^\infty u(0,y)dy=M,
\end{equation}
and more precisely, that the following Lyapunov functional $L$ converges to $0$,
\begin{equation}\label{Lyapounov_ss_crit}
L(\tau)=\int_0^\infty u(\tau,y)\log\left(\frac{u(\tau,y)}{G_\alpha(y)}\right) dy+\frac{\left(J(\tau)-\alpha(1-M)\right)^2}{2(1-M)}.
\end{equation}

In Proposition \ref{rate} we improve this result with an exponential decay to equilibrium. The proof of this result is done as follows. In the new variables, the equilibrium state is a gaussian, hence in a linear frame the natural tool would be a logarithmic Sobolev inequality established by Gross in \cite{gross} that we first recall, see \cite{Bobkov_frombrunn-minkowski} for a proof for instance. Although we are dealing here with a non linear problem this method will be fruitful. To do so, we apply this inequality to a suitable measure, namely $G_{u_0}(y)dy$. It would have been natural to apply a logarithmic Sobolev inequality to the measure $G_\beta (y)dy$ but a computation of the entropy dissipation with respect to the equilibrium state, $G_\beta$, leads to the Fisher information expressed with respect to  $G_{u_0}(y)dy$. Therefore a natural idea was to apply a logarithmic Sobolev inequality with respect to $G_{u_0}(y)dy$.

\begin{lemma} [Logarithmic Sobolev inequality] \label{LSI(1)} 
Let $\nu(x)dx=\exp(-V(x))dx$ be a measure with smooth density on $\R_+$. Assume that $V''(x)\geq 1$ then, for $u\geq 0$ satisfying $\int_{\R_+} u(x) dx=\int_{\R_+} \nu (x) dx$, we have
\begin{equation*}
\int_0^\infty u(x)\log \left(\frac{u(x)}{\nu(x)}\right)dx\leq \frac{1}{2}\int_0^\infty u(x)\left(\partial_x\left(\log\frac{u(x)}{\nu(x)}\right)\right)^2 dx.
\end{equation*}
\end{lemma}

On the first hand, let us define the map $\beta\mapsto C(\beta)$ by 
\begin{equation*}
C(\beta)\int_0^\infty \exp\left(-\beta y-\frac{y^2}{2}\right)dy=M,
\end{equation*}
and $V_\beta$ by $V_\beta(y)=\beta y +\frac{y^2}{2} + \log C(\beta)$. Such a function satisfies  $V_\beta''(y)= 1$ and
\begin{equation*}
C(\beta)\exp\left(-\beta y-\frac{y^2}{2}\right)=\exp\left(-V_\beta(y)\right)=G_\beta (y),
\end{equation*}  
hence lemma \ref{LSI(1)} can be applied and this yields that
\begin{equation}\label{LS_beta}
2H(u|G_\beta) \leq I(u|G_\beta).
\end{equation}

On the second hand, we recall  that the evolution of the relative entropy with respect to $G_\alpha$, see \cite{Siam_CHMV}, 
$$
\frac{d }{d \tau}  L(\tau)=-I(u|G_{u_0}) - \frac{\left(J(\tau)-u_0(\tau) (1-M)\right)^2}{1-M}.
$$
Furthermore,  the relative entropy can be decomposed as follows
$$
H(u|G_\alpha)=\int_0^\infty u(\tau,y) \log \left(\frac{u(\tau,y)}{G_\alpha(y)} \right)dy=\int_0^\infty u(\tau,y) \log \left(\frac{u(\tau,y) }{G_{u_0}(y)}\right)dy +\int_0^\infty u(\tau,y) \log \left(\frac{G_{u_0}(y)}{G_\alpha(y)}\right) dy.
$$
Recalling the definition of $G_{u_0}$ we deduce that
\begin{eqnarray*}
\int_0^\infty  u(\tau,y) \log \left(\frac{G_{u_0}(y)}{G_\alpha (y)}\right) dy & = & \int_0^\infty  u(\tau,y) \log \left(\frac{C_0}{\alpha}\right) dy+\int_0^\infty  u(\tau,y)  \log (\exp((\alpha-u_0(\tau))y)) dy\\
& = & M\log\left(\frac{C_0}{\alpha}\right)+(\alpha-u_0(\tau))J(\tau),
\end{eqnarray*}
where $C_0 = C(u_0)$. Therefore, 
$$
2H(u|G_\beta)=2H(u|G_{u_0})+2M\log\left(\frac{C_0}{\alpha}\right)+2(\alpha-u_0(\tau))J(\tau),
$$
from which we deduce that
\begin{eqnarray*}
\frac{d }{d \tau} L(\tau)+2L(\tau)& = & 2H(u|G_{u_0})-I(u|G_{u_0})+2M\log\left(\frac{C_0}{\alpha}\right) \\
& - & 2\alpha(u_0(\tau)-\alpha)(1-M) -(u_0(\tau)-\alpha)^2(1-M).
\end{eqnarray*}
Recalling next a logarithmic Sobolev inequality \eqref{LS_beta} together with $-(u_0(\tau)-\alpha)^2(1-M)\le 0$, it follows that 
\begin{equation*}
\frac{d }{d \tau} L(\tau)+2L(\tau)\le 2M\log\left(\frac{C_0}{\alpha}\right)-2\alpha(u_0(\tau)-\alpha)(1-M).
\end{equation*}
Now, it remains to evaluate the sign of the right-hand side term of the previous inequality, which we denote by $\Delta$. The  definition of $C_0$ first provides that
$$
\log\left(\frac{C_0}{\alpha}\right)=-\log \left(\frac{\int_0^\infty  \exp\left(-u_0(\tau) y-\frac{y^2}{2}\right)dy}{\int_0^\infty \exp\left(-\alpha y-\frac{y^2}{2}\right)dy}\right)=-\log \left(\int_0^\infty  \exp\left((\alpha-u_0(\tau))y\right)\frac{G_\alpha (y)}{M} dy\right),
$$
furthermore since 
\begin{eqnarray*}
-2\alpha(u_0(\tau)-\alpha)(1-M) & = & 2(\alpha-u_0(\tau))\int_0^\infty  \alpha y \exp\left(-\alpha y-\frac{y^2}{2}\right)dy\\
& = &2M\int_0^\infty \log(\exp\left((\alpha-u_0(\tau))y\right))\frac{G_\alpha(y)}{M}dy,
\end{eqnarray*}
from Jensen inequality it follows that
$$
\Delta
%=2M\log\frac{C_0}{\alpha}-2\alpha(u_0(\tau)-\alpha)(1-M)
=2M\left(\int_0^\infty \log(\exp((\alpha-u_0(\tau))y))\frac{G_\alpha(y)}{M} dy-\log \left(\int_0^\infty  \exp((\alpha-u_0(\tau))y)\frac{G_\alpha(y)}{M}dy \right)\right)\leq 0,
$$
hence
\begin{equation*}
\dfrac{d}{d \tau} L(\tau)+2L(\tau)\leq 0.
\end{equation*}
This achieves the proof of proposition \ref{rate}.
\begin{flushright}
$\square$
\end{flushright}

In order to obtain a rate of convergence for the $L^1$ norm we will use the  Csisz\'ar-Kullback  inequality, \cite{Csiszar},\cite{Kullback}.
\begin{proposition} [Csisz\'ar-Kullback  inequality]
For any non-negative functions  $f,g\in L^1(\R_+) $ such that $\int _{\R_+}  f(x) dx=\int  _{\R_+}  g(x) dx=M$, we have that 
\begin{equation}\label{CK}%\tag{Csisz\'ar-Kullback  inequality}
\|f-g\|_1^2\leq 2 M \int_0^\infty f(x)\log \left(\frac{f(x)}{g(x)}\right)dx.
\end{equation}
\end{proposition}
\begin{corollary}
The following inequalities hold true:
\begin{eqnarray*}
H(u|G_\alpha) & \leq & L(0)\exp(-2\tau),\\
\|u(\tau,y)- G_\alpha(y)\|_1 & \leq & \sqrt{2 M L(0)}\exp(-\tau),\\
\left|\left|n(t,x)-\frac{1}{\sqrt{1+2t}}G_\alpha \left(\frac{x}{\sqrt{1+2t}}\right)\right|\right|_1 & \leq & \sqrt{2 M L(0)}\frac{1}{(1+2t)^{3/2}}.
\end{eqnarray*}
\end{corollary}
%\begin{proof}
%The second point is mainly a consequence of Cszizar Kullback inequality \cite{Csiszar},\cite{Kullback}, for $f,g\geq 0$ satisfying $\int f=\int g$,
%$$
%\|f-g\|_1^2\leq 4\int f\log\frac{f}{g}.
%$$
%\end{proof}

\begin{remark}
The 'degradation' of the convergence as $M\rightarrow 1$ is not contained in the rate of convergence but in the intial value $L(0)$. Indeed, the correction term to the entropy contains the factor $(1-M)$ in the denominator. Therefore $L(0)$ may become very large as $M$ tends to  $1$. Hence, even if the rate does not depend on the mass,  the evolution of  $L(0)$ as $M \to 1$  affects the convergence quality. 
\end{remark}

\section{Critical mass and trend to equilibrium}\label{sec:M=1}
%We briefly remind the results for the critical case obtained in \cite{Siam_CHMV}. 
In the critical case, if $J(0)<\infty$, using \eqref{dJ}, we first notice that the first momentum $J$ is conserved.  Moreover the stationary states to \eqref{main_0} are given by the one parameter family: $n_\infty(0)\exp(-n_\infty(0)x)$. Hence $J_\infty=\frac{1}{n_\infty(0)}$ and there is only one  equilibrium state: $\nu_\alpha (x)=\alpha \exp(-\alpha x)$ with $\alpha =\frac{1}{J(0)}$. We also recall the formal computation of the time evolution of the relative entropy:
\begin{equation}\label{eq:critic_entropy_dissipation}
\frac{d}{dt}H(n|\nu_\alpha)=-\int_0^\infty n(t,x) (\partial_x\log n(t,x)+n(t,0))^2dx=-I(n|\nu_0),
\end{equation}
where we have used the notation 
\begin{equation*}
\label{eq:nu0}
\nu_0(t,x)dx=n(t,0)\exp(-n(t,0)x)dx.
\end{equation*}
In \cite{Siam_CHMV}, under the hypothesis that the second momentum is finite, it has been proved that the solution to \eqref{main_0} converges in relative entropy to $\nu_\alpha$. Here, we improve the convergence result by precising the speed of convergence. First,  we give a rate in the case of initially finite second momentum. Then, when the third momentum is initially finite, using the HWI inequality we improve the rate of convergence.

\subsection{The Lyapunov functional approach}
As we saw in the sub-critical case, a logarithmic Sobolev inequality is a powerful tool to deal with Gaussian measure. However, in the critical case, the stationary state is an exponentially decreasing measure but non-Gaussian. In order to use a logarithmic Sobolev inequality, we consider a Lyapunov functional approach. Indeed, corrective terms will bring the construction of a targeted Gaussian measure. %We will proceed by analysis and synthesis in order to find a valid expression for the corrective terms. 
We give a speed of convergence for relative entropy with the following theorem.
\begin{theorem}\label{th:critical1}
Assume that $\int_0^\infty x^2n^0(x) dx<+\infty$ and that $H(n^0|\nu_\alpha)<+\infty$, then 
\begin{equation*}
H(n|\nu_\alpha)\leq \frac{1}{\sqrt{1+2t}} \left(H(n^0|\nu_\alpha)+\int_0^\infty \frac{x^2}{2}n^0 (x) dx \right).
\end{equation*}
\end{theorem}
\begin{proof}
Let $c>0$ be a differentiable function on $\mathbb{R}_+$. Let us consider the nonnegative Lyapunov functional
\begin{equation*}
F(t)=H(n|\nu_\alpha)+c(t)\int_0^\infty \frac{x^2}{2}n(t,x)dx.
\end{equation*}
Using equation \eqref{eq:critic_entropy_dissipation}, we have that 
\begin{equation*}
\frac{d}{dt}F(t)=-I(n|\nu_0) +c(t)\frac{d}{dt} \left(\int_0^\infty \frac{x^2}{2}n(t,x) dx\right)+c'(t)\int_0^\infty \frac{x^2}{2}n(t,x) dx.
\end{equation*}
We need to bring up a Gaussian measure in Fisher information in order to use a logarithmic Sobolev inequality. An easy computation shows that 
\begin{equation}\label{second_moment}
\frac{d}{dt} \left(\int_0^\infty \frac{x^2}{2}n(t,x) dx \right)=\int_0^\infty \frac{x^2}{2} \partial_x\left(n(t,x) \partial_x \left(\log\frac{n(t,x)}{\nu_0(t,x)}\right)\right) dx=-\int_0^\infty x n(t,x)\partial_x \left(\log\frac{n(t,x)}{\nu_0(t,x)} \right) dx.
\end{equation}
The previous equality \eqref{second_moment} leads us to regroup the terms as follows
\begin{eqnarray}
\frac{d}{dt}F(t) & = &-\int_0^\infty n(t,x) \left[\left(\partial_x\left(\log\frac{n(t,x)}{\nu_0(t,x)}\right)\right)^2 +c(t)x\partial_x\left(\log\frac{n(t,x)}{\nu_0(t,x)}\right)+c(t)^2\frac{x^2}{4}\right] dx \nonumber \\
& & + \left(c'(t)+\frac{c(t)^2}{2}\right)\int_0^\infty\frac{x^2}{2}n(t,x) dx  \nonumber \\
& = & -\int_0^\infty n(t,x) \Bigg[\left(\partial_x\left(\log\frac{n(t,x)}{\nu_0(t,x)}\right)\right)^2+2\left(\partial_x \log \exp\left(\frac{c(t) x^2}{4}\right) \right)\left(\partial_x\left(\log\frac{n(t,x)}{\nu_0(t,x)}\right)\right) \nonumber \\
& & + \left(\partial_x \log \exp\left(\frac{c(t) x^2}{4}\right) \right)^2\Bigg] dx + \left(c'(t)+\frac{c(t)^2}{2}\right)\int_0^\infty\frac{x^2}{2}n(t,x) dx.\label{der_de_f}
\end{eqnarray}
We define the Gaussian measure $G_{0,c(t)/2}(t,x)dx$ by
\begin{equation}\label{def_gaussienne}
G_{0,c(t)/2}(t,x)dx =\dfrac{\nu_0(t,x) \, \exp\left(-\frac{c(t) x^2}{4}\right)}{\int_0^\infty \nu_0(t,x) \, \exp\left(-\frac{c(t) x^2}{4}\right)dx}dx.
\end{equation}
Recalling the definition of the Fisher information together with \eqref{def_gaussienne},  equality \eqref{der_de_f} rewrites as
\begin{eqnarray*}
\frac{d}{dt}F(t)=-I(n|G_{0,c(t)/2})+\left(c'(t)+\frac{c(t)^2}{2}\right)\int_0^\infty\frac{x^2}{2}n(t,x) dx.
\end{eqnarray*}
Using next a logarithmic Sobolev inequality \eqref{LSI(1)} for the Gaussian measure $G_{0,c(t)/2}$, it follows that
\begin{equation*}
I(n|G_{0,c(t)/2}) \geq c(t) H(n|G_{0,c(t)/2}),
\end{equation*}
hence 
\begin{equation}\label{dF}
\frac{d}{dt}F(t)\leq -c(t) H(n|G_{0,c(t)/2})+\left(c'(t)+\frac{c(t)^2}{2}\right)\int_0^\infty\frac{x^2}{2}n(t,x) dx.
\end{equation}
In order to have a rate of convergence, we want to make appear the Lyapunov functional on the right-hand side term and we use that
\begin{equation*}
H(n|G_{0,c(t)/2})=H(n|\nu_\alpha)+\int_0^\infty n\log \left(\frac{\nu_\alpha}{G_{0,c(t)/2}}\right)%H(\nu_\alpha|G_{0,c(t)/2}),
\end{equation*}
which allows rewriting \eqref{dF} as
\begin{equation}\label{dF2}
\frac{d}{dt}F(t)\leq -c(t) F(t)-c(t) \int_0^\infty n\log \left(\frac{\nu_\alpha}{G_{0,c(t)/2}}\right)+\left(c'(t)+ c(t)^2\right)\int_0^\infty\frac{x^2}{2}n(t,x) dx+ c(t)^2\int_0^\infty\frac{x^2}{4}n(t,x) dx.
\end{equation}
On the other hand,  since  $\int_{\R_+} n(t,x) dx=1$, we deduce that
\begin{eqnarray}\label{eq:controlG}
\int_0^\infty n(t,x) \log \left(\frac{\nu_\alpha(x)}{G_{0,c(t)/2}(t,x)} \right) dx & = & \log\left(\alpha\int_0^\infty  \exp\left(-n(t,0)x-\frac{c(t)x^2}{4}\right)dx\right) \nonumber \\ 
& + &(n(t,0)-\alpha)\int_0^\infty xn(t,x) dx + \frac{c(t)}{4}\int_0^\infty x^2n(t,x) dx,
\end{eqnarray}
hence this provides a control on $\int_0^\infty n\log \left(\frac{\nu_\alpha}{G_{0,c(t)/2}}\right)$.
Jensen inequality gives a first control on the following terms 
\begin{eqnarray*}
\log\left(\alpha\int_0^\infty  \exp\left(-n(t,0)x-\frac{c(t)x^2}{4}\right) dx\right) & = & \log\left(\int_0^\infty  \exp\left((\alpha-n(t,0))x-\frac{c(t)x^2}{4}\right) \nu_\alpha (x) dx\right) \\
& \geq & (\alpha-n(t,0)) \int_0^\infty x\nu_\alpha(x) dx-\frac{c(t)}{4}\int_0^\infty x^2\nu_\alpha(x) dx.
\end{eqnarray*}
Recalling the conservation of the first momentum, $\int_{\R_+} xn(t,x) dx=\int_{\R_+} x\nu_\alpha(x) dx$, \eqref{eq:controlG} simplifies  as
$$
\int_0^\infty n\log \left(\frac{\nu_\alpha}{G_{0,c(t)/2}}\right)
\geq \frac{c(t)}{4} \left(\int_0^\infty x^2n(t,x) dx-\int_0^\infty x^2\nu_\alpha(x) dx\right)=c(t)\left(\int_0^\infty \frac{x^2}{4}n(t,x) dx-\frac{1}{2\alpha^2}\right),
$$
coming back to \eqref{dF2},  
\begin{equation*}
\frac{d}{dt}F(t) \leq -c(t) F(t)+(c'(t)+c(t)^2)\int_0^\infty \frac{x^2}{2}n(t,x) dx +\frac{c(t)^2}{2\alpha^2},
\end{equation*}
hence, choosing $c(t)$ to satisfy 
\begin{equation}\label{defc}
c'(t)=-c(t)^2\left(1+\frac{1}{\alpha^2\int_0^\infty x^2 n(t,x) dx}\right),
\end{equation}
 we have a rate of convergence on the time evolution of $F$
\begin{equation}\label{eq:rate}
\frac{d}{dt}F(t)\leq -c(t) F(t),
\end{equation}
Using next Cauchy-Schwarz inequality on the measure $n(t,x) dx$, we see that 
\begin{equation*}
\int_0^\infty  x^2 n(t,x) dx =\left(\int_0^\infty n(t,x) dx\right) \left(\int_0^\infty  x^2 n(t,x) dx\right) \geq \left(\int_0^\infty  x n(t,x) dx \right)^2 = \frac{1}{\alpha^2}.
\end{equation*} 
hence a first lower bound on $c(t)$ defined by \eqref{defc}:
\begin{equation*}
c'(t)\geq -2c(t)^2,\quad c(t)\geq \frac{1}{1+2t}.
\end{equation*} 
For sake of simplicity, we choose $c(0)=1$. Using Gronwall lemma in \eqref{eq:rate}, we obtain that
\begin{equation}\label{eq:Gronwall1}
F(t)\leq F(0)\exp\left(-\int_0^t c(s)ds\right) \leq F(0)\exp\left(-\frac{1}{2}\log(1+2t)\right)=\frac{F(0)}{\sqrt{1+2t}},
\end{equation} 
Since $H(n|\nu_\alpha) \leq F(t)$, this achieves the proof. 
\begin{flushright}
$\square$
\end{flushright}
\end{proof}

From theorem \ref{th:critical1} and Csisz\'ar-Kullback  inequality \eqref{CK}, we immediately deduce the $L^1$ convergence 
\begin{equation*}
\|n-\nu_\alpha\|_1^2 \leq 2 H(n|\nu_\alpha) \leq \frac{C_0^2}{\sqrt{1+2t}},
\end{equation*}
where $C_0=\sqrt{2F(0)}=\sqrt{2 H(n^0|\nu_\alpha)+\int_0^\infty x^2 n^0(x)dx}$.
\begin{corollary}
Assume that $\int_0^\infty x^2n^0(x) dx<+\infty$ and that $H(n^0|\nu_\alpha)<+\infty$. For any time $t>0$, we have
\begin{eqnarray*}
\|n-\nu_\alpha\|_1 & \leq & \frac{C_0}{(1+2t)^\frac{1}{4}},
\end{eqnarray*}
where $C_0$ was previously defined.
\end{corollary}
During the proof of Theorem \ref{th:critical1}, we could have found a better lower bound for $c$, defined by \eqref{defc}, and consequently a better rate of convergence. 
\begin{proposition}
Assume that $\int_0^\infty x^2n^0(x) dx<+\infty$ and that $H(n^0|\nu_\alpha)<+\infty$. For all $\beta< 2/3$, there exists a constant $C$ and a time $t_\beta$ such that the following inequality holds true for any time $t\geq t_\beta$
\begin{eqnarray*}
\|n-\nu_\alpha\|_1 & \leq & \frac{C}{(1+t)^\frac{\beta}{2}}.
\end{eqnarray*}
\end{proposition}
\begin{proof}
We have a better lower bound of the second moment with the following lemma.
\begin{lemma} With the previous definition of $n$ and $\nu_\alpha$, the following inequality holds true
\begin{equation*} 
\liminf_{t \to +\infty} \int_0^\infty x^2 n(t,x) dx \geq \frac{2}{\alpha^2}.
\end{equation*}
\end{lemma}
\begin{proof}
Let $(t_k)_k$ be the following sequence
\begin{equation*} 
\liminf_{t \to +\infty} \int_0^\infty x^2 n(t,x) dx = \lim_{k \to +\infty} \int_0^\infty x^2 n(t_k,x) dx.
\end{equation*}
Recalling Theorem \ref{th:1D}, the convergence $L^1$ of $n$ towards $\nu_\alpha$ holds true. Thus, we can find a sub-sequence $(t_{k_p})_p$ of $(t_k)_k$ such that $n(t_{k_p}, .)$ converges towards $\nu_\alpha$ almost everywhere.
We define $$u_p(x)=\min (n(t_{k_p},x), \nu_\alpha(x)) \leq \nu_\alpha(x).$$ Since $\int_0^\infty x^2 u_p(x) dx \leq  \frac{2}{\alpha^2}$ and $x^2 u_p(x) \rightarrow x^2 \nu_\alpha(x)$ for almost all $x$, we can use Fatou's lemma
\begin{equation*}
\int_0^\infty x^2\nu_\alpha(x) dx  =  \int_0^\infty \liminf_{p \to +\infty} x^2 u_p(x) dx \leq \liminf_{p \to +\infty} \int_0^\infty x^2 u_p(x) dx \leq \liminf_{t \to +\infty} \int_0^\infty x^2 n(t,x) dx.
\end{equation*}
\begin{flushright}
$\square$
\end{flushright}
\end{proof}
This lemma provides  an upper bound on the right hand side term in \eqref{defc} hence, 
%\begin{equation*}
%\limsup_{t \rightarrow +\infty}\frac{1}{\alpha^2\int_0^\infty x^2n(t,x) dx }\leq \frac{1}{2}.
%\end{equation*}
by definition of the $\limsup$, for all $\eta>3/2$, we obtain that
\begin{equation*}
\inf_{s \geq 0} \left(\sup_{t\geq s} \frac{1}{\alpha^2\int_0^\infty x^2 n(t,x) dx} \right)< \eta -1.
\end{equation*}
Finally, there exists a time $t_\eta$ such that for $t\geq t_\eta$ with $c(t_\eta) >0$, 
\begin{equation*}
c'(t)\geq -\eta c(t)^2,\quad c(t)\geq \frac{1}{c(t_\eta)^{-1}+\eta(t-t_\eta)},\quad \int_{t_\eta}^t c(s)ds\geq \frac{1}{\eta}\log(c(t_\eta)^{-1}+\eta (t-t_\eta)).
\end{equation*}
In the same way as in \eqref{eq:Gronwall1}, we use the lower bound on $c$ to conclude: for any time $t\geq t_\eta$
\begin{equation*}
F(t)\leq F(t_\eta)\frac{1}{(c(t_\eta)^{-1}+\eta (t-t_\eta))^\frac{1}{\eta}}.
\end{equation*}
For all $\beta = \frac{1}{\eta}< 2/3$, the previous inequality proves that $(1+t)^{\beta}F(t)$ is bounded for $t\geq t_\beta =t_\eta$. 
\begin{equation*}
(1+t)^{\beta}F(t)\leq F(t_\eta)\frac{(1+t)^{\beta}}{(c(t_\beta)^{-1}+\frac{1}{\beta} (t-t_\beta))^\beta} \leq C.
\end{equation*}
We conclude this proof by using Csisz\'ar-Kullback inequality \eqref{CK}. 
\begin{flushright}
$\square$
\end{flushright}
\end{proof}

\subsection{A better rate}
In the previous paragraph, in order to use a logarithmic Sobolev inequality, we have constructed a Gaussian measure with corrective terms. In this paragraph we directly use the HWI inequality firstly established in \cite{Otto_villani} (see also \cite{TOT}), which is adapted for exponentially decreasing measure, and we improve the speed of convergence by controlling the second momentum. As we have seen in the Lyapunov functional approach, the rate of convergence could have been improved by using better estimations on the second momentum. The Wasserstein distance appears in the HWI inequality and we can control this distance with the second momentum. We need a third momentum uniformly bounded in time for such control on the second momentum. This is the following result. 
\begin{theorem}\label{lem:rate-1}
Assume that $\sup_{t\geq 0} \int_0^\infty x^3n(t,x) dx<+\infty$, then, there exist $t_0>0$, $C_1>0$ and $C_2>0$ such that for any time $t\geq t_0$ 
\begin{equation*}
H(n|\nu_\alpha)\leq \frac{1}{C_1+C_2 \, t}.
\end{equation*}
\end{theorem}
\begin{proof}
We break the proof into several lemmas. We start by recalling the HWI inequality obtained in \cite{Otto_villani}, see \cite{TOT} for instance. This inequality binds both Fisher information, entropy and Wasserstein distance and can be applied to exponentially decreasing measure.
\begin{lemma} [HWI inequality]
Let $\nu(x)dx=\exp(-V(x))dx$ be a measure with smooth density on $\mathbb{R}_+$, with $V''(x)\geq 0$ and $\int V(x) \exp(-V(x))dx < + \infty$. Then, for $n\geq0$ with finite momentums up to order 2, we have
\begin{equation*}
H(n|\nu) \leq W_2(n,\nu) \sqrt{I(n|\nu)}.
\end{equation*}
\end{lemma}
In the particular case of an exponential distribution $\nu_0 (t,x)dx=n(t,0)\exp(-n(t,0)x)dx$, the HWI inequality reads as
$$\frac{d}{dt}H(n|\nu_\alpha)=-I(n|\nu_0)\leq -\frac{H^2(n|\nu_0)}{W^2_2(n,\nu_0)}. $$
This inequality becomes very powerful once it is associated with the following inequalities.
\begin{lemma}\label{lem:HWItools} With the previous definition of $n$, $\nu_0$ and $\nu_\alpha$, the following inequalities hold true
\begin{enumerate}
\item
$H(n|\nu_0)\geq H(n|\nu_\alpha)$,
\item
$W_2(n,\nu_0)^2\leq 2 \int_0^\infty x^2n(t,x) dx + \frac{4}{n(t,0)^2}$.
\end{enumerate}
\end{lemma}
\begin{proof}
We compare the two terms of the first inequality 
\begin{equation*}
H(n|\nu_0)-H(n|\nu_\alpha)=\int_0^\infty n(t,x) \log\left(\frac{\nu_\alpha (x)}{\nu_0 (x)}\right) dx=\log \left(\frac{\alpha}{n(t,0)}\right)+(n(t,0)-\alpha)\int_0^\infty xn(t,x) dx.
\end{equation*}
For all $X >0$, we recall that $X-1 \geq \log X$ then
\begin{equation*}
H(n|\nu_0)-H(n|\nu_\alpha)=-\log \left(\frac{n(t,0)}{\alpha}\right)+\left(\frac{n(t,0)}{\alpha}-1 \right) \geq 0, \mbox{ with }\int_0^\infty xn(t,x) dx=\frac{1}{\alpha}.
\end{equation*}
Using the definition of the Wasserstein distance \eqref{defWassersteinDirac}, the second inequality is just a consequence of the triangle inequality
\begin{eqnarray*}
W_2(n,\nu_0) & \leq & W_2(n,\delta_0) + W_2(\delta_0,\nu_0) =  \sqrt{\int_0^\infty x^2n(t,x) dx} + \frac{\sqrt{2}}{n(t,0)},
\end{eqnarray*}
together with $(a+b)^2 \leq 2 a^2 + 2 b^2$ for $a,b \in \R$. 
\begin{flushright}
$\square$
\end{flushright}
\end{proof}

The Wasserstein distance is then controlled by the second momentum and $n(t,0)$. Recalling Proposition \ref{control}, for $t\geq \frac{\log 2}{(\alpha \pi)^2}=t_0$, we have a lower bound on $n(t,0)$: 
\begin{equation*}
\frac{4}{n(t,0)^2} \leq K.
\end{equation*}
Furthermore using Lemma \ref{lem:HWItools}, for any time $t\geq t_0$, we see that 
\begin{equation*}
\frac{d}{dt}H(n|\nu_\alpha)\leq -\frac{H^2(n|\nu_\alpha)}{K+2\int_0^\infty x^2n(t,x) dx},
\end{equation*} 
hence after time integration, 
\begin{equation}\label{eq:HWI_csqce}
\frac{1}{H(n|\nu_\alpha)}\geq \frac{1}{H(n(t_0,.)|\nu_\alpha)} + \int_{t_0}^t \frac{1}{K+2\int_0^\infty x^2n(t,x) dx}.
\end{equation} 
Using that $X \mapsto \frac{1}{X}$ is convex on $\mathbb{R}_+^*$ together with Jensen inequality, for the probability measure $\frac{1}{t-t_0} 1_{[t_0,t]} (s) ds$, it follows that
\begin{equation}\label{eq:ineqJensen}
\frac{1}{\int_0^\infty (K+2\int_0^\infty x^2n(t,x) dx) \frac{1}{t-t_0} 1_{[t_0,t]} (s) ds} \leq \int_0^\infty \frac{1}{K+2\int_0^\infty x^2n(t,x) dx} \frac{1}{t-t_0} 1_{[t_0,t]} (s) ds.
\end{equation}
therefore combining \eqref{eq:HWI_csqce} and \eqref{eq:ineqJensen},
\begin{equation}\label{eq:HWI_csqce2}
\frac{1}{H(n|\nu_\alpha)}\geq \frac{1}{H(n(t_0,.)|\nu_\alpha)} +\frac{t-t_0}{K+\frac{2}{t-t_0} \int_{t_0}^t\int_0^\infty x^2n(s,x) dx ds} .
\end{equation}
Thus,  if the second momentum is enough controlled, we obtain a better speed of convergence than in the Lyapunov functional approach.
\begin{remark}
If $\int_0^\infty x^2n(t,x) dx$ is uniformly bounded in time, then equation (\ref{eq:HWI_csqce2}) gives 
\begin{equation*}
\frac{1}{H(n|\nu_\alpha)}\geq \frac{1}{H(n(t_0,.)|\nu_\alpha)}+C \, (t-t_0).
\end{equation*}
\end{remark}
However, even if we do not control the second momentum, we can describe the behaviour of its Cesaro mean. We can then end with a final lemma
\begin{lemma}\label{lem:moment3}
Assume that $\sup_{t\geq 0} \int_0^\infty x^3n(t,x) dx<+\infty$, then for any time $t > t_0$
\begin{equation*}
\frac{1}{t-t_0}\int_{t_0}^t\int_0^\infty x^2n(s,x) dx ds \leq \frac{1}{3\alpha (t-t_0)}\left(\int_0^\infty x^3 n(t_0,x)dx-\frac{3\alpha}{4}\left(\int_0^\infty x^2n(t_0,x)dx\right)^2\right)+\frac{2}{\alpha^2}
\end{equation*}
\end{lemma}
\begin{proof}
We start by recalling a simple consequence of Holder inequality
\begin{equation}\label{eq:ineqHolder}
\left(\int_0^\infty x^2n(t,x) dx\right)^2 \leq \left(\int_0^\infty xn(t,x) dx \right) \left(\int_0^\infty x^3 n(t,x) dx \right)=\frac{1}{\alpha}\int_0^\infty x^3 n(t,x) dx, 
\end{equation}
with $\int_{\R_+} x n(t,x) dx=\frac{1}{\alpha}$ then we differentiate and this leads to 
\begin{equation}\label{eq:diff3mmt}
\frac{d}{dt} \left(\int_0^\infty x^3 n(t,x) dx \right) =\int_0^\infty x^3 \partial_x\left(\partial_x n(t,x)+ n(t,0) n(t,x)\right) dx=\frac{6}{\alpha}-3n(t,0)\int_0^\infty x^2n(t,x) dx.
%=\frac{6}{\alpha}-3\alpha\int_0^\infty x^2n(t,x) dx +\frac{3\alpha}{4}\times 2\left(2-2\frac{n(t,0)}{\alpha}\right)\int_0^\infty x^2n(t,x) dx,
\end{equation}
In the same way, we differentiate the second momentum
\begin{equation}\label{eq:diff2mmt}
\frac{d}{dt} \left(\int_0^\infty x^2n(t,x) dx \right)=\int_0^\infty x^2 \partial_x\left(\partial_x n(t,x)+ n(t,0) n(t,x)\right) dx=2-2\frac{n(t,0)}{\alpha}.
\end{equation}
Recalling \eqref{eq:diff2mmt},  $n(t,0)$ can be rewritten as
\begin{equation*}
n(t,0)=\alpha-\frac{\alpha}{2}\frac{d}{dt} \left(\int_0^\infty x^2 n(t,x) dx\right),
\end{equation*}
allowing us to rewrite equation \eqref{eq:diff3mmt} as 
\begin{equation*}
\frac{d}{dt} \left[\int_0^\infty x^3 n(t,x) dx \right]=\frac{6}{\alpha}-3\alpha\int_0^\infty x^2n(t,x) dx+\frac{3\alpha}{4}\frac{d}{dt} \left[\left(\int_0^\infty x^2 n(t,x) dx\right)^2\right].
\end{equation*}
Finally, we have obtained
$$
\frac{d}{dt}\left(\int_0^\infty x^3 n(t,x) dx-\frac{3\alpha}{4}\left(\int_0^\infty x^2n(t,x) dx\right)^2\right)=\frac{6}{\alpha}-3\alpha\int_0^\infty x^2n(t,x) dx,
$$
hence, after time integration, 
\begin{eqnarray*}
& & \left(\int_0^\infty x^3 n(t,x) dx-\frac{3\alpha}{4}\left(\int_0^\infty x^2n(t,x) dx\right)^2\right)+3\alpha \int_{t_0}^t\int_0^\infty x^2n(s,x) dx ds \\
& = & \left(\int_0^\infty x^3 n(t_0,x)dx-\frac{3\alpha}{4}\left(\int_0^\infty x^2n(t_0,x)dx\right)^2\right)+\frac{6(t-t_0)}{\alpha},
\end{eqnarray*}
therefore, recalling \eqref{eq:ineqHolder}, we deduce that
$$
\frac{1}{t-t_0} \int_{t_0}^t\int_0^\infty x^2n(s,x) dx ds = \frac{1}{3 \alpha(t-t_0)}\left(\int_0^\infty x^3 n(t_0,x)dx-\frac{3\alpha}{4}\left(\int_0^\infty x^2n(t_0,x)dx\right)^2\right)+\frac{2}{\alpha^2}.
$$
\begin{flushright}
$\square$
\end{flushright}
\end{proof}
Consequently, using Lemma~\ref{lem:moment3} with \eqref{eq:HWI_csqce2}, it follows that
\begin{equation*}
\frac{1}{H(n|\nu_\alpha)}\geq \frac{1}{H(n(t_0,.)|\nu_\alpha)} +\frac{t-t_0}{K+\frac{C_0}{t-t_0}+\frac{2}{\alpha^2}}.
\end{equation*}
Taking the inverse, this achieves the proof of Theorem~\ref{lem:rate-1}. 
\begin{flushright}
$\square$
\end{flushright}
\end{proof}

%\begin{remark}
%Since we have the convergence $L^1$ of $n$ towards $\nu_\alpha$, we can improve theorem \ref{lem:rate-1} by getting better estimations on $W_2 (n,\nu_\alpha)$.
%\end{remark}

\medskip

\noindent{\em Acknowledgement: The authors want to warmly thank Vincent Calvez, without whom this work would not have existed.}

\bibliographystyle{Siam}

\end{document}